\documentclass[12pt,a4paper]{article}
\usepackage{amssymb}
\usepackage{mathrsfs}
\usepackage{amsmath}
\usepackage{amsfonts,amsthm,amssymb}
\usepackage{graphics}
\usepackage{graphicx}
\usepackage{epstopdf}
\usepackage[]{caption2}
\usepackage{color}
\usepackage{float}
\usepackage[margin=1in]{geometry}
\usepackage{chngcntr}
\usepackage{standalone} 
\theoremstyle{definition}
\newtheorem{definition}{Definition}
\newtheorem*{remark}{Remark}
\newtheorem{thm}{Theorem}  
\newtheorem{lem}[thm]{Lemma}

\begin{document}
	\captionsetup{font={footnotesize}}
	\title{A constructive characterization of uniformly $4$-connected graphs}
	
	\author{
		Xiang Chen$^a$, 
		\quad Shuai Kou$^a$,
		\quad Chengfu Qin$^b$,
		\quad Liqiong Xu$^c$, \\
		\quad Weihua Yang$^a$\footnote{Corresponding author. E-mail: ywh222@163.com,~yangweihua@tyut.edu.cn}\\
		\small $^a$Department of Mathematics, Taiyuan University of Technology, Taiyuan 030024, China\\  
		\small $^b$School of Mathematics and Statistics, Nanning Normal University, Nanning 530001, China\\
		\small $^c$School of Science, Jimei University, Xiamen 361021, China\\
	}
	
	\date{}
	
	\maketitle {\flushleft\bf Abstract:} {\small  A constructive characterization of the class of uniformly $4$-connected graphs is presented. The characterization is based on the application of graph operations to appropriate vertex and edge sets in uniformly $4$-connected graphs, that is, any uniformly $4$-connected graph can be obtained from $C_5^2$ or $C_6^2$ by a number of $\Delta_1^+$ or $\Delta_2^+$-operations to quasi $4$-compatible sets.}
	
	
	\section{Introduction}
	\hspace*{1.3em} In this paper, we only consider finite simple undirected graphs, with graph theoretic terminology used here                                                                                                                                                                                                                                                                                                                                                                                 following \cite{Bondy}.
	
	The class of $3$-connected graphs has been characterized by Tutte~\cite{Tutte} in his famous Wheel Theorem. A construction characterization of the class of minimally $3$-connected graphs was given by Dawes~\cite{Dawes}. In addition, the class of $4$-connected graphs has been characterized by Qin and Ding~\cite{Qin}. Yin~\cite{Yin} also gave a recursive construction method for $4$-connected graphs. Su et al.~\cite{Su} gave a construction method for $k$-connected graphs.
	
	A graph is said to be \emph{uniformly $k$-connected} if between any pair of vertices the maximum number of internally-disjoint paths is exactly $k$. The study of uniformly connected graphs originated with Beineke et al.~\cite{Beineke}, who first derived some properties of uniformly connected graphs and defined different operations to yield infinite families of uniformly connected graphs, but they by no means construct all such graphs. Subsequently, G\"{o}ring et al.~\cite{Goring} gave a construction characterization of the class of uniformly $3$-connected graphs, which requires the set of all $3$-connected and $3$-regular graphs as a starting set. Recently, Xu~\cite{Xu} improved the construction method for uniformly $3$-connected graphs, which requires only the graph $K_4$ as the staring set.

	In this paper, we will give a constructive characterization of the class of uniformly $4$-connected graphs. The Section $2$ of this paper presents several definitions and basic lemmas that will be used in the proof of the main result, which is determined in Section~$3$.

	\section{Preliminaries}
	
	\hspace*{1.3em} For a graph $G$, let $V(G)$ and $E(G)$ denote the vertex set and the edge set of $G$, respectively. Let $N_G(x)$ denote the \emph{neighborhood} of $x$ in $G$ where $x\in V(G)$, and let $V(e)$ denote the set of \emph{endpoints} of $e$ where $e\in E(G)$. For a $k$-connected graph $G$ and a minimum vertex cut $T$ of $G$, the union of the vertex sets of at least one but not of all components of $G-T$ is called a \emph{$T$-fragment}(or briefly a \emph{fragment}) of $G$. A fragment of $G$ is called an \emph{end} of $G$ if any of its proper subsets is not a fragment of $G$.
	
	\begin{definition}\label{def1}
		A path $p$ is a \emph{quasi $3$-circuit chording path}, if between the endpoints of a subpath $p^{\prime}$ of $p$, there exist three internally-disjoint paths except for $p^{\prime}$, and the three paths intersect $p$ only at the endpoints of $p^{\prime}$.
		A path is an \emph{$e^{+}$-quasi $3$-circuit chording path}, if it is not a quasi $3$-circuit chording path, but becomes one after adding the edge $e$.
		A \emph{quasi chord} of a cycle is a path $p$ which connects two non-adjacent vertices of the cycle where $p$ intersects the cycle only at the endpoints of $p$.   
	\end{definition}
	
	\begin{definition}\label{def2}
		Let $H$ be a $4$-connected graph. Let $X=\{a,b,c\}\subset V(H)$ and $Y=\{d,e,f\}\subset V(H)$ such that $\vert X\cap Y\vert \leqslant 2$ and $E(H[X]), E(H[Y])\neq \emptyset$. The \emph{$\Delta_1^{+}$-operation} and \emph{$\Delta_2^{+}$-operation} are defined as follows:

		\noindent \emph{$\Delta_1^{+}$-operation}. $\Delta_1^{+}(H)=(H-E_{X})+x+\{xp|p\in X\cup \{y\}\}$, where $x\notin V(H)$, $y\in V(H)\setminus X$, 
		
		\noindent\hspace{2.55cm} $E_X \subseteq E(H[X])$ and $E_X \neq \emptyset$ such that $\kappa(H-E_{X})=\kappa(\Delta_1^{+}(H)-x)\geqslant 3$.
		
		\noindent \emph{$\Delta_2^{+}$-operation}. $\Delta_2^{+}(H)=(H-E_{X}-E_{Y})+x+y+xy+\{xp|p\in X\}+\{yq|q\in Y\}$, where 
		
		\noindent\hspace{2.55cm} $x,y\notin V(H)$, $E_X \subseteq E(H[X])$, $E_Y \subseteq E(H[Y])$ and $E_X, E_Y \neq \emptyset$ such that
		
		\noindent\hspace{2.55cm} $\kappa(H-E_{X}-E_{Y})=\kappa(\Delta_2^{+}(H)-x-y)\geqslant 2$, moreover, if $\kappa(H-E_{X}-E_{Y})=$
		
		\noindent\hspace{2.55cm} $\kappa(\Delta_2^{+}(H)-x-y)= 2$, any end of $H-E_{X}-E_{Y}$ contains both a vertex 
		
		\noindent\hspace{2.55cm} in $X$ and a vertex in $Y$. 
	\end{definition}
	
	\begin{definition}\label{def3}
		Let $H$ be a $4$-connected graph. Let $X=\{a,b,c\}\subset V(H), y\in V(H)\setminus X$ and $Y=\{d,e,f\}\subset V(H)$ such that $\vert X\cap Y\vert \leqslant 2$ and $E(H[X]), E(H[Y])\neq \emptyset$. Let $E_X \subseteq E(H[X])$ and $E_Y \subseteq E(H[Y])$ such that $E_X, E_Y \neq \emptyset$. Let $K_X$ denote a triangle defined by the vertices $a,b,c$ and let $K_Y$ denote a triangle defined by the vertices $d,e,f$. Then a set $S$ of vertices and edges of $H$ is \emph{quasi $4$-compatible} if it conforms to one of the following two types:
		
		\noindent Type $(1)$. $S=\{v,e\vert v\in X\cup \{y\}, e\in E_X\}$, where no $uv$-path is a quasi $3$-circuit chording 
		
		\noindent\hspace{1.75cm} path of $H-E_X$ for $u\in X, v=y$ or $u, v\in V(e)$ where $u\neq v$ and $e\in E(K_X)\setminus E_X$.
		
		\noindent Type $(2)$. $S=\{v,e\vert v\in X\cup Y, e\in E_X\cup E_Y\}$, where $S$ conforms to the following two 
		
		\noindent\hspace{1.75cm} conditions simultaneously:
		
		\noindent\hspace{1.75cm} (i) No $uv$-path is a quasi $3$-circuit chording path of $H-E_X-E_Y$ for $u\in X-$ 
		
		\noindent\hspace{2.35cm} $Y, v\in Y-X$ or $u, v\in V(e)$ where $u\neq v$ and $e\in (E(K_X)\setminus E_X)\cup (E(K_Y)\setminus$
		
		\noindent\hspace{2.35cm} $E_Y)$, moreover, if $\vert X\cap Y\vert =2$, no $uv$-path is a quasi chord of any cycle of 
		
		\noindent\hspace{2.35cm} $H-E_X-E_Y$ for $u,v\in X\cap Y$ where $u\neq v$.
		
		\noindent\hspace{1.75cm} (ii) No $uv$-path is an $e^{+}$-quasi $3$-circuit chording path of $H-E_X-E_Y$ for $u$,
		
		\noindent\hspace{2.45cm} $v\in V(e_1), e\in (E(K_{Y})\setminus E(H[Y]))\cup E_{Y}$ where  $u\neq v$ and $e_{1}\in E(K_{X})\setminus E_X$,
		
		\noindent\hspace{2.45cm} or $u,v\in V(e_2), e\in (E(K_{X})\setminus E(H[X]))\cup E_{X}$ where $u\neq v$ and $e_{2}\in E(K_{Y})\setminus$ 
		
		\nopagebreak
		\noindent\hspace{2.45cm} $E_Y$.
	\end{definition}
	
	Yin~\cite{Yin} defined removable edges in $4$-connected graphs and proved the following results.

	\begin{definition}\label{def4}
		An edge $e$ of a $4$-connected graph $G$ is said to be a \emph{removable edge} if $G\ominus e$ is still $4$-connected, where $G\ominus e$ denotes the graph obtained from $G$ by deleting e to get $G-e$, and for any endpoint of $e$ with degree $3$ in $G-e$, say $x$, deleting $x$ and then adding edges between any pair of non-adjacent vertices in $N_{G-e}(x)$.
	\end{definition}
	
	\begin{lem}(\cite{Yin})\label{lem1}
		Let $G$ be a $4$-connected graph of order at least $7$  and $e=xy\in E(G)$. Then $e$ is non-removable if and only if there exists $S\subseteq V(G)$ with $\vert S\vert=3$ such that $G-e-S$ has exactly two components $A, B$ with $\vert A\vert \geqslant 2$ and $\vert B\vert \geqslant 2$, moreover $x\in A, y\in B$.
	\end{lem} 
	
	\begin{lem}(\cite{Yin})\label{lem2}
		A $4$-connected graph without removable edges is either $C_5^2$ or $C_6^2$.
	\end{lem}
	
	For $k$-connected graphs, the following result was given in Su et al.~\cite{Su}.
	\begin{lem}(\cite{Su})\label{lem3}
		Let $H$ be a $k$-connected graph with $k\geqslant 3$, let $X=\{x_1,x_2,\cdots,x_{k-1}\}\subset V(H)$ and $Y=\{y_1,y_2,\cdots,y_{k-1}\}\subset V(H)$.
		
		(i) If $H[X]\cong K_{k-1}$, then $G_X=(H-E_{X})+x+\{xx_i|i=1,2,\cdots,k-1\}+xy$ is $k$-connected if and only if $\kappa(H-E_{X})=\kappa(G_X-x)\geqslant k-1$, where $E_X\subseteq E(H[X]), x\notin V(H), y\in V(H)-X$;
		
		(ii) If $H[X]\cong K_{k-1}$ and $H[Y]\cong K_{k-1}$, then $G_{XY}=(H-E_{X}-E_{Y})+x+y+xy+\{xx_i|i=1,2,\cdots,k-1\}+\{yy_i|i=1,2,\cdots,k-1\}$ is $k$-connected if and only if $\vert X\cap Y\vert \leqslant k-2, \kappa(H-E_{X}-E_{Y})=\kappa(G_{XY}-x-y)\geqslant k-2$, and, if $\kappa(H-E_{X}-E_{Y})=\kappa(G_{XY}-x-y)= k-2$, any end of $H-E_{X}-E_{Y}$ contains both a vertex in $X$ and a vertex in $Y$, where $E_X\subseteq E(H[X]), E_Y\subseteq E(H[Y])$, $x,y\notin V(H)$.
		
	\end{lem}
	
	For removable edges of $4$-connected graphs, we have the following two conclusions. 
	
	\begin{lem}\label{lem4}
		Let $G$ be a $4$-connected graph with $\vert V(G)\vert \geqslant 7$ and let $R$ be a triangle of $G$. If $x\in V(R)$ with $d_G(x)\geqslant 5$, then there exists $y\in V(R)\setminus\{x\}$ such that $xy$ is removable.
	\end{lem}
	
	\begin{proof}
		Let $V(R)\setminus\{x\}=\{y_1,y_2\}$. Suppose $xy_1$ and $xy_2$ are both non-removable. Then by Lemma~\ref{lem1}, there exists $S\subseteq V(G)$ with $\vert S\vert=3$ such that $G-xy_1-S$ has exactly two components $A, B$ with $\vert A\vert \geqslant 2$ and $\vert B\vert \geqslant 2$, by symmetry we assume that $x\in A$ and $y_1\in B$. Moreover, there exists $T\subseteq V(G)$ with $\vert T\vert=3$ such that $G-xy_2-T$ has exactly two components $C, D$ with $\vert C\vert \geqslant 2$ and $\vert D\vert \geqslant 2$, by symmetry we assume that $x\in C$ and $y_2\in D$.
		
		Since $R$ is a triangle of $G$, we have $y_1\in T$ and $y_2\in S$. Now we observe that $x\in A\cap C, y_1\in B\cap T, y_2 \in S\cap D$. Let $X_1=(S\cap C)\cup (S\cap T)\cup (A\cap T), X_2=(A\cap T)\cup (S\cap T)\cup (S\cap D), X_3=(B\cap T)\cup (S\cap T)\cup (S\cap D), X_4=(S\cap C)\cup (S\cap T)\cup (B\cap T)$.
		
		{\bf Claim 1.} $\vert X_1\vert \geqslant 3$.
		
		For otherwise, since $d_G(x)\geqslant 5$, we have $A\cap C\neq \{x\}$. Note that $N_G((A\cap C)\setminus \{x\})=X_1\cup \{x\}$, which implies that $X_1\cup \{x\}$ is a $k$-cut of $G$ where $k\leqslant 3$, contradicting that $G$ is $4$-connected.
		
		{\bf Claim 2.} $A\cap T\neq \emptyset, S\cap C\neq \emptyset$.
		
		We show that $A\cap T\neq \emptyset$. For otherwise, by Claim $1$, we have $\vert (S\cap C)\cup (S\cap T)\vert \geqslant 3$, contradicting that $\vert S\vert=3$. By the similar arguments, we have $S\cap C\neq \emptyset$.
		
		{\bf Claim 3.} $S\cap T\neq \emptyset$.
		
		For otherwise, we assume that $S\cap T= \emptyset$. If $\vert S\cap C\vert =2$, then $\vert S\cap D\vert =1$ and $\vert X_2\vert, \vert X_3\vert \leqslant 3$, which implies that $A\cap D=B\cap D=\emptyset$, contradicting that $\vert D\vert \geqslant 2$. Hence we can assume that $\vert S\cap C\vert =1$. By Claim 1, we observe that $\vert A\cap T\vert =2$, then $\vert X_3\vert=3$ and $\vert X_4\vert=2$, which implies that $B\cap D=B\cap C=\emptyset$, contradicting that $\vert B\vert \geqslant 2$.
		
		Now we complete the proof of Lemma~\ref{lem4}. By Claims $2$ and $3$, we have that $\vert S\cap C\vert=\vert S\cap T\vert=\vert S\cap D\vert=\vert A\cap T\vert=\vert B\cap T\vert=1$. This implies that $A\cap D=B\cap D=\emptyset$, contradicting that $\vert D\vert \geqslant 2$.
	\end{proof}
	
	\begin{lem}\label{lem5}
		Let $G$ be a uniformly $4$-connected graph except for $C_5^2$ and $C_6^2$. Let $e=xy$ be a removable edge of $G$ with $d_G(x)=4$ and $d_G(y)\geqslant5$. Suppose $N_G(x)\setminus \{y\}=X=\{a,b,c\}$ such that $\vert E(G[X])\vert\leqslant1$. Let $H=G\ominus e$ such that $H$ is not uniformly $4$-connected with $\vert V(H)\vert \geqslant 7$. Then $X\subset V(H)$ and there exists an edge $e\in E(H[X])\setminus E(G[X])$ such that $e$ is a removable edge of $H$.
	\end{lem}
	
	\begin{proof}
		By Definition \ref{def4}, we have $X\subset V(H)$ obviously. It remains to show that there exists an edge $e$ of $H$ with the required properties. It follows that $H$ is $4$-connected and $H[X]\cong K_3$ by Definition \ref{def4}. Since $H$ is $4$-connected and not uniformly $4$-connected, it follows that there exists a pair of vertices which are $5$-connected in $H$, say $u$ and $v$. We can clearly observe that $u,v\in V(G)$ as well.
		
		At first we consider the case that $\vert E(G[X])\vert=1$, say $E(G[X])=\{ac\}$. Observe that $u$ and $v$ are $5$-connected in $H$, which implies that there exist at least five internally-disjoint paths between $u$ and $v$ in $H$. Now we show that $b\in \{u,v\}$. By contradiction, we assume that $b\notin \{u,v\}$. If neither $ab$ nor $bc$ lies on any of these paths, then $u$ and $v$ are $5$-connected in $G$ obviously, contradicting that $G$ is uniformly $4$-connected. If either $ab$ or $bc$ lies on one of these paths, say $P_1$, where $ab\in E(P_1)$ and $bc\notin E(P_1)$, then we can replace the edge $ab$ by $axb$, as the vertex $x$, the edges $ax$ and $bx$ do not exist in $H$. Thus we can define at least five internally-disjoint paths between $u$ and $v$ in $G$, a contradiction. Consequently, both $ab$ and $bc$ lie on one of these paths. Since $b\notin \{u,v\}$, it follows that both $ab$ and $bc$ lie on the same path, as these paths are internally-disjoint. Thus, we can replace the subpath $abc$ by $axc$, as the vertex $x$, the edges $ax$ and $bx$ do not exist in $H$. Hence we can define at least five internally-disjoint paths between $u$ and $v$ in $G$, a contradiction. Consequently, $b\in \{u,v\}$, which implies $d_H(b)\geqslant 5$. Since $\vert V(H)\vert \geqslant 7$ and $H[X]\cong K_3$, $ba$ or $bc$ is removable in $H$ by Lemma \ref{lem4}. Observe that $ba,bc\in E(H[X])\setminus E(G[X])$, thus there exists an edge $e\in E(H[X])\setminus E(G[X])$ such that $e$ is a removable edge of $H$.
		
		Next we consider the case that $E(G[X])=\emptyset$. By the arguments similar to those of the above case, we can obtain that $\{a,b,c\}\cap\{u,v\}\neq\emptyset$. Then there exists a vertex $w\in \{a,b,c\}$ such that $d_H(w)\geqslant 5$. Hence there exists an edge $e\in \{ab,ac,bc\}$ is removable in $H$. Observe that $\{ab,ac,bc\}=E(H[X])\setminus E(G[X])$, thus there exists an edge $e\in E(H[X])\setminus E(G[X])$ such that $e$ is a removable edge of $H$.
	\end{proof}
	
	The following lemma follows straightforwardly from Lemma~\ref{lem3}.
	\begin{lem}\label{lem6}
		Let $H$ be a $4$-connected graph. Let $G$ be constructed by applying $\Delta_1^+$ or $\Delta_2^+$-operation to a set $S$ of vertices and edges of $H$. Then $G$ is $4$-connected. 
	\end{lem}

	\section{The main result}
	\hspace*{1.3em} We will now demonstrate that quasi $4$-compatibility is both necessary and sufficient to guarantee that applying the $\Delta_1^+$ or $\Delta_2^+$-operation to a quasi $4$-compatible set in a uniformly $4$-connected graph results in a graph that remains uniformly $4$-connected.
	
	\begin{lem}\label{lem7}
		Let $H$ be a uniformly $4$-connected graph. Let $G$ be constructed by applying the $\Delta_1^+$-operation to a set $S$ of vertices and edges of $H$. Then $G$ is uniformly $4$-connected if and only if $S$ is a quasi $4$-compatible set in $H$. 
	\end{lem}

	\begin{proof}
		First, we suppose that $G$ is a uniformly $4$-connected graph, and that $S$ is not a quasi $4$-compatible set in $H$. We will give a detailed proof only for $H[X]\cong K_3$ and $E_X=\{ac\}$. We observe that $S=\{a,b,c,y,ac\}$(see Fig. \ref{Fig1}) by Definition \ref{def3}. Since $S$ is not quasi $4$-compatible in $H$, there must exist a $uv$-path where $u\in X, v= y$ or $u=a, v= b$ or $u=b, v=c$ which is a quasi $3$-circuit chording path in $H-ac$. However, in $G$, the endpoints of the subpath of that quasi $3$-circuit chording path are $5$-connected, which contradicts that $G$ is uniformly $4$-connected. Therefore, $S$ must be a quasi $4$-compatible set in $H$.
		
		\begin{figure}[H]
			\begin{center}
				\input{1.TpX}
			\end{center}
		\end{figure} 
		
		\begin{figure}[H]
			\begin{center}
				\input{2.TpX}
			\end{center}
		\end{figure}
		
		Conversely, suppose that $S$ is a quasi $4$-compatible set in $H$, and that $G$ is not uniformly $4$-connected. We also only discuss the case that $H[X]\cong K_3$ and $E_X=\{ac\}$. We observe that $S=\{a,b,c,y,ac\}$ by Definition \ref{def4}. By Lemma~\ref{lem6}, $G$ is $4$-connected. By assumption, $G$ is not uniformly $4$-connected. Then $G$ contains a pair of vertices $u$ and $v$ with connectivity more than $4$. Since $d_G(x)=4$, $x\notin \{u,v\}$. Since $H$ is uniformly $4$-connected, $u$ and $v$ are not $5$-connected in $H$, it must be true that for any set of five internally-disjoint paths connecting $u$ and $v$ in $G$, one path must contain the vertex $x$, say $P_1$(see Fig. \ref{Fig2}). Now we will consider $\{u,v\}$ and $N_{P_1}(x)$, where $N_{P_1}(x)$ denotes $N_{G}(x)$ in $P_1$.

		\textbf{Case 1.} $\{u,v\}\neq \{a,b\}, \{b,c\}$.
		Obviously, $N_{P_1}(x)\neq \{a,b\}, \{b,c\}$, for otherwise, we can replace $axb(bxc)$-subpath of $P_1$ with the edge $ab(bc)$ in $G$, respectively, contradicting that $P_1$ contains $x$. If $N_{P_1}(x)=\{a,c\}$, then we can replace $axc$-subpath of $P_1$ with the edge $ac$ in $H$, which contradicts that $H$ is uniformly $4$-connected. If $y\in N_{P_1}(x)$, then it must be seen that the $ay$-path, the $by$-path or the $cy$-path is a quasi $3$-circuit chording path in $H-ac$, which contradicts the quasi $4$-compatibility of $S$ in $H$.
		
		\textbf{Case 2.} $\{u,v\}=\{a,b\}$ or $\{u,v\}=\{b,c\}$. Without loss of generality, we can assume that $\{u,v\}=\{a,b\}$. Obviously, for any set of five internally-disjoint paths connecting $a$ and $b$ in $G$, one path must be the edge $ab$. For otherwise, this contradicts that $H$ is uniformly $4$-connected. Analogously, we claim that $N_{P_1}(x)\neq \{b,c\}$. If $N_{P_1}(x)=\{a,c\}$ or $y\in N_{P_1}(x)$, the proof is similar to the Case $1$. If $N_{P_1}(x)=\{a,b\}$, then the edge $ab$ will be a quasi $3$-circuit chording path in $H-ac$, a contradiction.
		
		Therefore, $G$ is uniformly $4$-connected if $S$ is a quasi $4$-compatible set in $H$.
	\end{proof}
	
	\begin{lem}\label{lem8}
		Let $H$ be a uniformly $4$-connected graph. Let $G$ be constructed by applying the $\Delta_2^+$-operation to a set $S$ of vertices and edges of $H$. Then $G$ is uniformly $4$-connected if and only if $S$ is a quasi $4$-compatible set in $H$. 
	\end{lem}
	
	\begin{proof}
		First, we suppose that $G$ is uniformly $4$-connected, and that $S$ is not a quasi $4$-compatible set in $H$. We will give a detailed proof only for $E(H[X])=\{ac,bc\}, H[Y]\cong K_3, X\cap Y=\emptyset, E_X=\{bc\}$ and $E_Y=\{df\}$. We observe that $S=\{a,b,c,d,e,f,bc,df\}$ by Definition \ref{def3}. Since $S$ is not quasi $4$-compatible in $H$, there exists a $uv$-path where $u\in X, v\in Y$ or $u=a, v\in \{b,c\}$ or $u=e, v\in \{d,f\}$ which is a quasi $3$-circuit chording path in $H-bc-df$, or there exists an $ab$-path or $ac$-path which is a $df^+$-quasi $3$-circuit chording path in $H-bc-df$, or there exists a $de$-path or $ef$-path which is an $ab^+$ or $bc^+$-quasi $3$-circuit chording path in $H-bc-df$. Observe that in $G$, the endpoints of the subpath of that quasi $3$-circuit chording path or $df^+$, $ab^+$ or $bc^+$-quasi $3$-circuit chording path are $5$-connected, which contradicts that $G$ is uniformly $4$-connected. Therefore, $S$ must be a quasi $4$-compatible set in $H$. 
		
		Conversely, suppose that $S$ is a quasi $4$-compatible set in $H$, and that $G$ is not uniformly $4$-connected. We also only discuss the case that $E(H[X])=\{ac,bc\}, H[Y]\cong K_3, X\cap Y=\emptyset, E_X=\{bc\}$ and $E_Y=\{df\}$. We observe that $S=\{a,b,c,d,e,f,bc,df\}$ by Definition \ref{def3}. By the similar arguments, $G$ contains a pair of vertices $u$ and $v$ with connectivity more than $4$, and $x,y\notin \{u,v\}$. Since $H$ is uniformly $4$-connected, $u$ and $v$ are not $5$-connected in $H$. Let $\kappa_G(u,v)$ denote the connectivity of $u,v$ in $G$ and let $\kappa_H(u,v)$ denote the connectivity of $u,v$ in $H$. The following several cases arise, depending on $\kappa_G(u,v)$. If $\kappa_G(u,v)\geqslant 7$, then $\kappa_H(u,v)\geqslant 5$, which contradicts that $H$ is uniformly $4$-connected. Hence this case can never arise.
		
		\textbf{Case 1.} $\kappa_G(u,v)=6$.
		We claim that for any set of six internally-disjoint paths connecting $u$ and $v$ in $G$, one path must contain the vertex $x$ and another path must contain the vertex $y$, say $P_1, P_2$, respectively, where $P_1\neq P_2$. Two subcases arise, depending on $\{u,v\}, N_{P_1}(x)$ and $N_{P_2}(y)$, where $N_{P_1}(x)$ denotes $N_{G}(x)$ in $P_1$ and $N_{P_2}(y)$ denotes $N_{G}(y)$ in $P_2$.
		
		Subcase (a). $\{u,v\}\neq\{a,c\}, \{d,e\}, \{e,f\}$.
		Obviously, $N_{P_1}(x)\neq \{a,c\}, N_{P_2}(y)\neq \{d,e\},\{e,f\}$. If $N_{P_1}(x)=\{a,b\}$ or $\{b,c\}$, $N_{P_2}(y)=\{d,f\}$, then we can replace $dyf$-subpath of $P_2$ with the edge $df$ in $H$. Thus, $u$ and $v$ are $5$-connected in $H$, which contradicts that $H$ is uniformly $4$-connected.
		
		Subcase (b). $\{u,v\}=\{a,c\}$. Obviously, for any set of six internally-disjoint paths connecting $a$ and $c$ in $G$, one path must be the edge $ac$. For otherwise, this contradicts that $\kappa_G(u,v)=6$. Then it is clear that the edge $ac$ will be a quasi $3$-circuit chording path in $H-bc-df$, which contradicts the quasi $4$-compatibility of $S$ in $H$.  
		
		Subcase (c). $\{u,v\}=\{d,e\}$ or $\{e,f\}$.
		The proof of this subcase is similar to Subcase (b).
		
		\textbf{Case 2.} $\kappa_G(u,v)=5$.
		We claim that for any set of five internally-disjoint paths connecting $u$ and $v$ in $G$, there exists a path which must contain the vertices $x$ or $y$. Similarly, several subcases arise.
		
		Subcase (a). $\{u,v\}\neq\{a,c\},\{d,e\},\{e,f\}$.
		If only the vertex $x$ is on one of five internally-disjoint paths, say $P_1$, while $y$ is not. Then $N_{P_1}(x)\neq \{a,c\}$, clearly. If $N_{P_1}(x)=\{a,b\}$, then it is clear that there exists an $ab$-path which is a quasi $3$-circuit chording path in $H-bc-df$, which contradicts the quasi $4$-compatibility of $S$ in $H$. If $N_{P_1}(x)=\{b,c\}$, then we can replace $bxc$-subpath of $P_1$ with the edge $bc$ in $H$, which implies that $u$ and $v$ are $5$-connected in $H$, a contradiction. The proof is similar for that only the vertex $y$ is on one of five internally-disjoint paths, while $x$ is not. Now we will consider that the vertices $x$ and $y$ are both on one of five internally-disjoint paths. If the vertices $x$ and $y$ are on the same path, then there exists a $uv$-path where $u\in X, v\in Y$ such that the path is a quasi $3$-circuit chording path in $H-bc-df$, a contradiction. If the vertices $x$ and $y$ are on different paths, say $P_1$, $P_2$, respectively. Obviously, $N_{P_1}(x)\neq \{a,c\}$ and $N_{P_2}(y)\neq \{d,e\},\{e,f\}$. If $N_{P_1}(x)=\{a,b\}$, $N_{P_2}(y)=\{d,f\}$, then it must be seen that there exists an $ab$-path which is a $df^+$-quasi $3$-circuit chording path in $H-bc-df$, a contradiction. If $N_{P_1}(x)=\{b,c\}$, $N_{P_2}(y)=\{d,f\}$, then we can replace $bxc$-subpath of $P_1$ and $dyf$-subpath of $P_2$ with the edges $bc$ and $df$ in $H$, respectively, which implies that $u$ and $v$ are $5$-connected in $H$, a contradiction.
		
		Subcase (b). $\{u,v\}=\{a,c\}$.
		Obviously, for any set of five internally-disjoint paths connecting $a$ and $c$ in $G$, one path must be the edge $ac$. For otherwise, this contradicts that $\kappa_G(u,v)=5$. We first consider the case that only the vertex $x$ is on one of five internally-disjoint paths, say $P_1$, while $y$ is not. If $N_{P_1}(x)=\{a,b\}$ or $\{b,c\}$, the proof is similar to Subcase (a). If $N_{P_1}(x)=\{a,c\}$, then it is clear that the edge $ac$ is a quasi $3$-circuit chording path in $H-bc-df$, which contradicts the quasi $4$-compatibility of $S$ in $H$. The proof is similar for that only the vertex $y$ is on one of five internally-disjoint paths, while $x$ is not. Now we will consider that the vertices $x$ and $y$ are both on one of five internally-disjoint paths. If the vertices $x$ and $y$ are on the same path, the proof is similar to Subcase (a). If the vertices $x$ and $y$ are on different paths, say $P_1$, $P_2$, respectively. Analogously, $N_{P_2}(y)\neq \{d,e\},\{e,f\}$. If $N_{P_1}(x)=\{a,b\}$ or $\{b,c\}$, $N_{P_2}(y)=\{d,f\}$, the proof is similar to Subcase (a). If $N_{P_1}(x)=\{a,c\}, N_{P_2}(y)=\{d,f\}$, then it must be seen that the edge $ac$ will be a $df^+$-quasi $3$-circuit chording path in $H-bc-df$, a contradiction.
		
		Subcase (c). $\{u,v\}=\{d,e\}$ or $\{e,f\}$.
		The proof of this subcase is similar to Subcase (b).
		
		Therefore, $G$ is uniformly $4$-connected if $S$ is a quasi $4$-compatible set in $H$.
	\end{proof}
	
	\begin{remark}
		For Lemmas \ref{lem7} and \ref{lem8}, although we have not provided individual proofs for all the cases presented here, they possess common characteristics and the proofs are similar. In particular, for the case of $\vert X\cap Y\vert =2$ in lemma \ref{lem8}, we add a condition to the definition of the quasi $4$-compatible sets in Type (2) of Definition \ref{def3}, namely that no $uv$-path is a quasi chord of any cycle of $H-E_X-E_Y$ for $u,v\in X\cap Y$ where $u\neq v$.
	\end{remark}
	
	We are now prepared to establish that any uniformly $4$-connected graph, except for $C_5^2$ and $C_6^2$, can be constructed by applying $\Delta_1^+$ or $\Delta_2^+$-operation to a smaller uniformly $4$-connected graph.
	
	\begin{thm}\label{thm9}
		Let $G$ be a graph. Then $G$ is uniformly $4$-connected if and only if one of the following holds:
		
		(i) $G\cong C_5^2$ or $C_6^2$,
		
		(ii) there exists a uniformly $4$-connected graph $G^\prime$, $\vert G^\prime\vert < \vert G\vert$, such that $G$ can be constructed by applying $\Delta_1^+$ or $\Delta_2^+$-operation to a quasi $4$-compatible set in $G^\prime$.  
	\end{thm}
	
	\begin{proof}
		Obviously, $G$ is uniformly $4$-connected if $G\cong C_5^2$ or $C_6^2$, and by Lemma~\ref{lem7} and Lemma~\ref{lem8}, $G$ is uniformly $4$-connected if $G$ satisfies the condition (ii).
		
		Conversely, let $G$ be a uniformly $4$-connected graph except for $C_5^2$ and $C_6^2$. We will show that there exists an appropriate uniformly $4$-connected graph $G^\prime$. Since $G\not\cong C_5^2, C_6^2$, by Lemma~\ref{lem2}, $G$ contains a removable edge $e$, and let $e=xy$ and $H_1=G\ominus e$. The following three cases arise, depending on the degree of $x$ and $y$ in $G$.
		
		\textbf{Case 1.} $d_G(x)\geqslant 5, d_G(y)\geqslant 5$.
		By Lemma~\ref{lem2}, $H_1=G-xy$ is $4$-connected. Obviously there exist five internally-disjoint paths between $x,y\in V(G)$, which contradicts that $G$ is uniformly $4$-connected. Hence this case can never arise.   
		
		\textbf{Case 2.} $d_G(x)=4, d_G(y)\geqslant 5$ or $d_G(x)\geqslant 5, d_G(y)=4$.
		Without loss of generality, we can assume that $d_G(x)=4, d_G(y)\geqslant 5$. Let $N_G(x)\setminus \{y\}=X=\{a,b,c\}$. If $G[X]\cong K_3$, then $H_1=G-x$. Since $H_1$ is $4$-connected, there exist at least four internally-disjoint paths between two vertices $a,b\in V(H_1)$. Observe that the vertex $x$, the edges $ax$ and $bx$ do not exist in $H_1$, consequently, we can define at least five internally-disjoint paths between $a$ and $b$ in $G$ by adding $axb$ to these paths, which contradicts that $G$ is uniformly $4$-connected. Thus, $G[X]\not\cong K_3$. 
		
		If $\vert E(G[X])\vert=2$, say $E(G[X])=\{ab,bc\}$. To see that $H_1=G^\prime$, that is, $H_1$ is also uniformly $4$-connected, suppose for the sake of contradiction that there exist five internally-disjoint paths between two vertices $u,v\in V(H_1)$. If $ac$ does not lie on any of these paths, obviously, $u,v$ will be $5$-connected in $G$, a contradiction. If $ac$ lies on one of these paths, we may safely replace the edge $ac$ by $axc$, which implies that $u$ and $v$ are $5$-connected in $G$, a contradiction. Consequently, $H_1$ is uniformly $4$-connected. Since we can construct $G$ by applying $\Delta_1^+$-operation to $S=\{a,b,c,y,ac\}$ in $H_1$, we can conclude by Lemma~\ref{lem7} that $S$ is a quasi $4$-compatible set in $H_1$. If $\vert E(G[X])\vert=1$, say $E(G[X])=\{ac\}$. Two subcases must be considered.
		
		Subcase (a). $H_1$ is uniformly $4$-connected.
		Then $G$ can be constructed from the uniformly $4$-connected graph $H_1$ by applying $\Delta_1^+$-operation to $S=\{a,b,c,y,ab,bc\}$ in $H_1$, which is therefore a quasi $4$-compatible set in $H_1$ by Lemma~\ref{lem7}.
		
		Subcase (b). $H_1$ is not uniformly $4$-connected. We first assume that $\vert V(H_1)\vert\geqslant7$. Then by Lemma \ref{lem5}, there exists a removable edge $e\in \{ab,bc\}$ in $H_1$. Without loss of generality, we can assume that $ab$ is removable in $H_1$ and let $H_2=H_1\ominus ab$. In order to find the graph $G^\prime$, we need to initially state the following facts. By the proof of Lemma \ref{lem5}, we know that the vertex $b$ must be in the $5$-connected vertex pair of $H_1$ and $d_{H_1}(b)\geqslant5$. Now we show that if there exist at least five internally-disjoint paths between two vertices $b,v\in V(H_1)$, then for any set of five internally-disjoint paths connecting $b$ and $v$ in $H_1$, one path must contain the edge $ba$ and another path must contain the edge $bc$. For otherwise, we can obviously define at least five internally-disjoint paths between $b$ and $v$ in $G$, contradicting that $G$ is uniformly $4$-connected.
		
		Since $ab$ is removable in $H_1$ and $d_{H_1}(b)\geqslant5$, we now consider $d_{H_1}(a)$. If $d_{H_1}(a)\geqslant5$, then $H_2=H_1-ab$. Now we show that $H_2=G^\prime$, that is, $H_2$ is uniformly $4$-connected. By contradiction, we suppose that there exist five internally-disjoint paths between two vertices $u_2,v_2\in V(H_2)$. We can clearly observe that these paths also exist in $H_1$, thus $b\in \{u_2,v_2\}$ by the previous fact. However, since the edge $ab\notin E(H_2)$, it follows that $ab$ can not occur in any of these paths. This contradicts the previous fact. Hence $H_2=G^\prime$. Since we can construct $G$ by applying $\Delta_1^+$-operation to $S=\{a,b,c,y,bc\}$ in $H_2$, we can conclude by Lemma~\ref{lem7} that $S$ is a quasi $4$-compatible set in $H_2$.
		
		Now we consider the case that $d_{H_1}(a)=4$. Let $N_{H_1}(a)\setminus\{b\}=X_1=\{m,n,c\}$. Then we need to consider $H_1[X_1]$. If $H_1[X_1]\cong K_3$, then $H_2=H_1-a$. Since $H_2$ is $4$-connected, there exist at least four internally-disjoint paths between two vertices $m,c\in V(H_2)$. As the vertex $a$, the edges $am$ and $ac$ do not exist in $H_2$, we can define at least five internally-disjoint paths between two vertices $m$ and $c$ in $H_1$ by adding $mac$ to these paths. This contradicts the previous fact that $b$ is in the $5$-connected vertex pair of $H_1$.
		
		If $\vert E(H_1[X_1])\vert\leqslant2$. Repeating the above discussion, we can either find the uniformly $4$-connected graph $G^\prime$, or we can identify at least one removable edge which is incident to a vertex in a $5$-connected vertex pair as long as the order of the corresponding graph is at least $7$. Since we consider finite graphs, this process must terminate. If we have already found the graph $G^\prime$, then we can construct $G$ from the uniformly $4$-connected graph $G^\prime$ by applying a series of $\Delta_1^+$-operations to the corresponding sets. Moreover, we can conclude by Lemma~\ref{lem7} that the corresponding sets are quasi $4$-compatible sets in the corresponding graphs.
		
		Now we will illustrate how to solve the problem in the case that if there exists a sequence of $4$-connected graphs $G,H_1,\cdots,H_n$ such that it conforms to the following three conditions simultaneously:
		
		(1) $G\ominus e_0=H_1$ and $H_i\ominus e_{i}=H_{i+1}$, where $e_0\in E(G)$ and $e_{i}\in E(H_i)$ for $1\leqslant i\leqslant n-1$,
		
		(2) $G$ is uniformly $4$-connected, and $H_1,H_2,\cdots,H_n$ are not uniformly $4$-connected,
		
		(3) $\vert V(H_n)\vert\leqslant6$.
		
		Since $H_n$ is $4$-connected, $\vert V(H_n)\vert\geqslant5$. If $\vert V(H_n)\vert=5$, then $H_n\cong K_5\cong C_5^2$. Hence we assume that $\vert V(H_n)\vert=6$. Since $H_n$ is $4$-connected and not uniformly $4$-connected, it follows that $H_n$ is isomorphic to one of the following three graphs, $Oct^+, K_6\setminus e$ and $K_6$. If $H_n$ is isomorphic to $Oct^+$ or $K_6\setminus e$, then performing the operation of removable edges one more time will lead to $C_5^2$. Now we assume that $H_n\cong K_6$, which implies that any pair of vertices in $H_n$ is $5$-connected. Let $V(H_n)=\{a,b,c,d,e,f\}$, then $d$ and $e$ are $5$-connected in $H_n$. Observe that $H_{n-1}\ominus e_{n-1}=H_{n}$. Let $V(e_{n-1})=\{x_{n-1},y_{n-1}\}$. According to the previously mentioned principle of taking removable edges, we know that at least one endpoint of $e_{n-1}$ has degree at least $5$, say $d_{H_{n-1}}(y_{n-1})\geqslant 5$. Now we consider $d_{H_{n-1}}(x_{n-1})$. If $d_{H_{n-1}}(x_{n-1})\geqslant 5$, then $H_n=H_{n-1}-e_{n-1}$, this contradicts that $H_n\cong K_6$. Hence $d_{H_{n-1}}(x_{n-1})=4$. We have $N_{H_{n-1}}(x_{n-1})\subset V(H_n)$ obviously. Without loss of generality, we suppose  $N_{H_{n-1}}(x_{n-1})\setminus \{y_{n-1}\}=\{a,b,c\}=X_{n-1}$ and $y_{n-1}\in \{d,e,f\}$. Then by the similar arguments, we know that
		$H_{n-1}[X_{n-1}]\not\cong K_3$. If $\vert E(H_{n-1}[X_{n-1}])\vert=2$, then $H_n$ is uniformly $4$-connected, a contradiction. If $\vert E(H_{n-1}[X_{n-1}])\vert\leqslant 1$, then there exists a vertex $w\in \{a,b,c\}$ such that $w$ is in a $5$-connected vertex pair of $H_n$, which contradicts that $d$ and $e$ are $5$-connected in $H_n$. Consequently, $H_n\not\cong K_6$.

		If $\vert E(G[X])\vert=0$, by the similar arguments, we can obtain that there exists an edge $e\in \{ab,ac,bc\}$ is removable in $H_1$. The remaining process proceeds analogously to the previous discussion.
		
		\textbf{Case 3.} $d_G(x)=4, d_G(y)=4$.
		Let $N_G(x)\setminus\{y\}=X=\{a,b,c\}, N_G(y)\setminus \{x\}=Y=\{d,e,f\}$. Note that $\vert X\cap Y\vert \leqslant 2$ as $G$ is $4$-connected. By the similar arguments to the Case $2$, we know that $G[X]\not\cong K_3, G[Y]\not\cong K_3$. If $G[X]\cong P_3$ and $G[Y]\cong P_3$, then $H_1=G\ominus e$ must be uniformly $4$-connected. Thus $G$ can be constructed from the uniformly $4$-connected graph $H_1$ by applying $\Delta_2^+$-operation to the corresponding set, which is therefore a quasi $4$-compatible set in $H_1$ by Lemma~\ref{lem8}. If $\vert E(G[X])\vert\leqslant2,\vert E(G[Y])\vert\leqslant2$, but $\vert E(G[X])\vert=2$ and $\vert E(G[Y])\vert=2$ do not hold simultaneously, the proof is similar to the Case $2$.		
	\end{proof}

\end{document}